\documentclass[12pt]{amsart}
\usepackage{amsmath, amssymb, amsthm}
\newtheorem{example}{Example}
\usepackage{geometry}
\usepackage{graphicx}
\usepackage{tikz}
\usepackage{mathrsfs}
\usetikzlibrary{matrix}
\usetikzlibrary{arrows.meta}
\usetikzlibrary{arrows.meta,positioning}
\newtheorem{thm}{Theorem}[section]

\newtheorem{cor}[thm]{Corollary}

\newtheorem{prop}[thm]{Proposition}

\newtheorem{lem}[thm]{Lemma}
\theoremstyle{definition}
  
\newtheorem{defn}[thm]{Definition}

\newtheorem{remark}{Remark}
\usepackage{tikz}
\usepackage{placeins}
\usepackage{subcaption}
\usepackage{comment}

\newcommand{\SL}{\mathrm{SL}}
\newcommand{\Sp}{\mathrm{Sp}}

\DeclareMathOperator{\sgn}{Sgn}

\DeclareMathOperator{\Pf}{Pf}
\DeclareMathOperator{\Hom}{Hom}

\usepackage{tikz}
\usepackage{array}

\usepackage[colorlinks=true, linkcolor=blue]{hyperref}
\usetikzlibrary{arrows.meta}
\usetikzlibrary{calc}

\title{Electrical Networks and Symplectic Invariants}
\author{David Kogan}
\date{}

\begin{document}

\maketitle

\begin{abstract}
Consider a finite planar graph with positive real edge weights
and designated boundary vertices, called nodes. Such a
graph is called a circular planar electrical network. A grove is a spanning forest in which every
component contains at least one node. The connected components of a grove determine a partition of the nodes. We relate weighted grove counts
to invariant theory for the symplectic group.

To a planar electrical network $G$, we associate an
$\Sp(2n)$-invariant tensor $Z_G$.  For $\Sp(2)=\SL(2)$, we expand $Z_G$ in the Temperley--Lieb basis indexed by noncrossing matchings and relate its coefficients to the Kenyon--Wilson grove formulas. For $\Sp(4)$, we give reduction rules for superpositions of two groves and prove that the tensors indexed by $3$-noncrossing matchings form a basis of the space of $\Sp(4)$-invariant tensors.
The coefficients of $Z_G$ in this basis are
weighted counts of reduced double groves, up to normalization.
\end{abstract}

\section{Introduction}

The spanning tree model is a classical probabilistic and combinatorial model
associated with an electrical network.  Kirchhoff's matrix-tree theorem \cite{kirchhoff1847ueber}
expresses the weighted sum of spanning trees as the
determinant of the graph Laplacian with one row and column removed. Let $G$ be a finite graph embedded in a disk, with positive
edge weights and distinguished vertices on the boundary,
called \emph{nodes}, labeled $1,2,\ldots,2m$ in
clockwise order. We call $G$ a \emph{network}.
A natural extension of the spanning-tree model to this setting is the
model of \emph{groves}: spanning forests of $G$ in which
every connected component contains at least one node.
Each grove $F$ determines a noncrossing partition $\sigma(F)$ of the nodes, with
two nodes in the same block precisely when they belong
to the same component. 

For each partition $\sigma$
of the nodes, the corresponding \emph{grove partition function}
is
\[
Z[\sigma]=\sum_{F:\,\sigma(F)=\sigma} w(F),
\]
where the sum is over groves of $G$ and $w(F)$ is the
product of the edge weights of $F$.
These partition functions can be computed using the
\emph{response matrix} $L$, which maps voltages at the nodes
to the resulting boundary currents. The entry $L_{ij}$ is the current flowing into the network at node $i$ when node
$j$ is held at unit voltage and all other nodes are held
at zero voltage. Write $Z_{\mathrm{unc}}:=Z[1|2|\dots|2m]$ for the partition
function of groves in which every node lies in a distinct
component. Kenyon and Wilson showed that the normalized
partition functions $Z[\sigma]/Z_{\mathrm{unc}}$ are
polynomials in the entries of the response matrix
\cite{kenyon2008combinatorics,kenyon2011boundary}.

A central problem in spanning tree and related models is to compute
boundary connection probabilities in the scaling limit. Explicit formulas for such limits are related to multiple
$\mathrm{SLE}_{\kappa}$ connection probabilities
\cite{kenyon2011boundary,liu2025uniform}. 

Related questions arise in the dimer model. Kenyon and Wilson
obtained analogous formulas for boundary connection probabilities
in the double-dimer model in terms of boundary measurements
\cite{kenyon2011boundary}. More generally, superpositions of
$r$ dimer configurations give rise to $\SL(r)$-invariant
tensors represented by planar graphs called webs
\cite{fraser2019dimers,douglas2024dimers}.
For the triple-dimer model, connection probabilities for reduced
$\SL(3)$ webs can be computed from the boundary measurement
matrix \cite{kenyon2023planar}.

We study weighted superpositions of $n$ groves with boundary
connectivities $\sigma_1,\ldots,\sigma_n$.
Each superposition determines a graph $H$ with edge multiplicities. Assign a copy of $W=\mathbb C^{2n}$, with the standard action of $\Sp(2n)$, to each boundary node. We associate to $H$ an $\Sp(2n)$-invariant multilinear
form $T_H$ on these boundary spaces, which we call its
coloring tensor.
We define the network tensor $Z_G$ by
\[
Z_G(v_1,\ldots,v_{2m})
:=
\operatorname{Pf}\!\left(
\bigl(-L_{ij}\omega(v_i,v_j)\bigr)_{i,j=1}^{2m}
\right),
\]
where $L$ is the response matrix of $G$, $v_i\in W$, $\operatorname{Pf}$ denotes the Pfaffian,
and $\omega$ is the standard symplectic form on $W$. We can view $Z_G$ as an element of $\Hom_{\Sp(2n)}(W^{\otimes 2m}, \mathbb C)$.
We show that $Z_G$ is a normalized weighted sum of
the coloring tensors $T_H$.

When \(n=1\), \(\Sp(2)\cong \SL(2)\), the Kenyon–Wilson grove formulas lead to an expansion of \(Z_G\) in the
Temperley–Lieb basis indexed by noncrossing matchings, with coefficients given by normalized grove
partition functions. Related connections between cycle-rooted spanning forests and
$\SL(2)$ connections were developed using vector-bundle
Laplacians in \cite{kenyon2011spanning}.

For $n=2$, we construct a basis of $\mathrm{Hom}_{\Sp(4)}((\mathbb C^4)^{\otimes{2m}}, \mathbb C)$. Our construction
adapts the double-grove reduction of Gao, Lam, and Xu
\cite{gao2025electrical} to groves whose boundary blocks
are singletons or pairs. The resulting graphs are indexed
by $3$-noncrossing matchings: perfect matchings of the
boundary nodes with no three pairwise crossing chords.
Replacing each crossing by a $4$-valent vertex gives
the corresponding planar graph. The coloring tensors and reduced double-grove partition
functions appearing below are defined precisely in Section~\ref{sec:section_2} and Section \ref{sec:sec_4}. Our main result is

\begin{thm}\label{thm:sp4_grove_expansion1}
Let $G$ be a planar network with $2m$ boundary
vertices, and suppose $n=2$. Then
\[
Z_G
=
\frac{1}{Z_{\mathrm{unc}}^2}
\sum_{\xi\in\mathrm{TC}(2m)} Z[\xi]\,T_\xi,
\]
where $\mathrm{TC}(2m)$ denotes the set of $3$-noncrossing
matching graphs on $2m$ boundary vertices and $T_\xi$ is the
associated coloring tensor. The coefficient $Z[\xi]$
is the reduced double-grove partition function: each
distinct uncolored superposition contributes its network
weight multiplied by the coefficient of $\xi$ in its
grove reduction.
\end{thm}

As a corollary, the tensors $T_\xi$ form a basis of
$\operatorname{Hom}_{\Sp(4)}(W^{\otimes 2m},\mathbb C)$. The same matchings index the bush basis of the degree-two
component of the grove algebra \cite{gao2025electrical}.
In our construction, the $3$-noncrossing matching graphs
arise by reducing superpositions on the primal network, while
in the bush basis construction, they arise on the medial graph.

It would be interesting to extend our construction to
incorporate all grove partition functions simultaneously,
beyond those indexed by partitions with singleton and
pair blocks. Further connections between electrical networks, symplectic
groups, and Lagrangian Grassmannians were studied in
\cite{lam2015electrical,lam2018electroid,bychkov2021electrical}.

Kuperberg constructed bases for invariant tensors for rank-two Lie
algebras using planar graphs called webs, together with
local graphical relations \cite{kuperberg1996spiders}. Kuperberg constructed a $\Sp(4)$
web basis \cite{kuperberg1996spiders} for $\mathrm{Hom}_{\Sp(4)}(W^{\otimes 2m}, \mathbb C)$. Using his $4$-valent
vertex, this web basis is indexed by the same
$3$-noncrossing matching diagrams. The coloring tensors in our
construction generally differ from the corresponding
web tensors, giving two bases of the same invariant space.

The proof combines the Kenyon--Wilson grove formulas with the pairing
expansion of the Pfaffian and the reduction theory of double groves.
Section~2 introduces the network tensor and its pairing expansion,
Section~3 treats the \(\SL(2)\) case, and Section~4 develops the
double-grove construction and proves the \(\Sp(4)\) expansion and basis
theorem.

\section{Electrical networks and the grove partition functions} \label{sec:section_2}

We recall the basic setup for circular planar electrical
networks; see, for example, \cite{curtis1998circular}. Let $G=(V,E)$ be a finite planar graph embedded in a disk, with
boundary vertices
\(
N=\{1,\ldots,2m\}
\) called \emph{nodes} 
lying on the boundary of the disk in cyclic order, and let
$I=V\setminus N$ denote the set of interior vertices.
Each edge $e\in E$ is assigned a positive conductance $c_e$.

The weighted graph Laplacian $\Delta$ is the $V\times V$ matrix
\[
\Delta_{ij}
=
\begin{cases}
\displaystyle\sum_{k\sim i}c_{ik}, & i=j,\\
-c_{ij}, & i\sim j,\\
0, & \text{otherwise}.
\end{cases}
\]
Split $\Delta$ into blocks according to $V=I\sqcup N$. The response matrix
(or Dirichlet-to-Neumann matrix) of the network is the Schur complement
\[
L=\Delta_{NN}-\Delta_{NI}\Delta_{II}^{-1}\Delta_{IN}.
\]
We assume that every connected component of $G$ contains a
boundary node, which implies that $\Delta_{II}$ is invertible. \(L_{ij}\) is the current flowing into the network at node \(i\) when boundary node \(j\) is held at one volt and all other boundary nodes are held at zero volts.

A \emph{grove} of $G$ is a spanning forest in which every connected
component contains at least one boundary node (but possibly no edges).  Its boundary
connectivity is the partition of $N$ induced by its connected
components.  For a partition $\tau$ of $N$, we write $Z[\tau]$ for
the weighted sum of groves with boundary connectivity $\tau$, and
we write $Z_{\mathrm{unc}}$ for the partition function of groves in
which every boundary vertex lies in a distinct component. If we consider the graph where all nodes are identified then applying the matrix-tree theorem gives $Z_{\mathrm{unc}}=\det \Delta_{II}$. See figure \ref{fig:grove-example}.

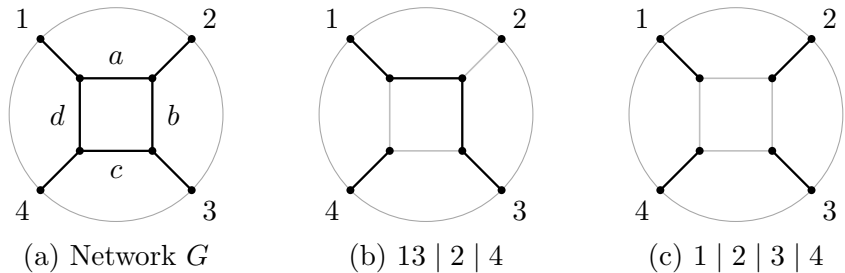
\begin{figure}[ht]
\centering
\begin{tikzpicture}[
    edge/.style={black,line width=.85pt},
    unused/.style={black!25,line width=.6pt},
    every node/.style={font=\small}
]
\foreach \panel/\offset in {0/0,1/4.1,2/8.2}{
\begin{scope}[xshift=\offset cm]

    \draw[black!35,line width=.4pt]
        (0,0) circle ({sqrt(2)});

    \draw[unused]
        (-1,1)--(-.48,.48)
        (1,1)--(.48,.48)
        (1,-1)--(.48,-.48)
        (-1,-1)--(-.48,-.48);
    \draw[unused]
        (-.48,.48)--(.48,.48)
        --(.48,-.48)--(-.48,-.48)--cycle;

    \ifnum\panel=0
        \draw[edge]
            (-1,1)--(-.48,.48)
            (1,1)--(.48,.48)
            (1,-1)--(.48,-.48)
            (-1,-1)--(-.48,-.48);
        \draw[edge]
            (-.48,.48)--(.48,.48)
            --(.48,-.48)--(-.48,-.48)--cycle;

        \node[above] at (0,.53) {$a$};
        \node[right] at (.53,0) {$b$};
        \node[below] at (0,-.53) {$c$};
        \node[left]  at (-.53,0) {$d$};

        \node at (0,-1.9) {(a) Network $G$};
    \fi

    \ifnum\panel=1
        \draw[edge]
            (-1,1)--(-.48,.48)--(.48,.48)
            --(.48,-.48)--(1,-1);
        \draw[edge]
            (-1,-1)--(-.48,-.48);

        \node at (0,-1.9) {(b) $13\mid2\mid4$};
    \fi

    \ifnum\panel=2
        \draw[edge]
            (-1,1)--(-.48,.48)
            (1,1)--(.48,.48)
            (1,-1)--(.48,-.48)
            (-1,-1)--(-.48,-.48);

        \node at (0,-1.9) {(c) $1\mid2\mid3\mid4$};
    \fi

    \fill (-1,1) circle (1.5pt);
    \fill (1,1) circle (1.5pt);
    \fill (1,-1) circle (1.5pt);
    \fill (-1,-1) circle (1.5pt);
    \fill (-.48,.48) circle (1.5pt);
    \fill (.48,.48) circle (1.5pt);
    \fill (.48,-.48) circle (1.5pt);
    \fill (-.48,-.48) circle (1.5pt);

    \node[above left]  at (-1,1) {$1$};
    \node[above right] at (1,1) {$2$};
    \node[below right] at (1,-1) {$3$};
    \node[below left]  at (-1,-1) {$4$};

\end{scope}
}
\end{tikzpicture}
\caption{An electrical network and two of its groves.
The inner edges have conductances $a,b,c,d$, as indicated,
and the four outer edges have conductance $1$.
The groves shown in (b) and (c) contribute weights $ab$
and $1$ to $Z[13\mid2\mid4]$ and $Z_{\mathrm{unc}}$,
respectively. Unused edges are shown in light gray. For this network $Z[13\mid2\mid4]=ab+cd+abc+abd+acd+bcd$.}
\label{fig:grove-example}
\end{figure}

Kenyon and Wilson \cite{kenyon2011boundary} showed that grove partition functions can be recovered from the response matrix. More precisely, the normalized grove partition functions \(Z[\tau]/Z_{\mathrm{unc}}\) are integer-coefficient polynomials in the matrix entries $L_{ij}$.

Let $W=\mathbb C^{2n}$ carry the standard symplectic form
$\omega$. We choose a symplectic basis
\(
e_1,\ldots,e_n,f_1,\ldots,f_n
\)
such that
\[
\omega(e_i,f_j)=\delta_{ij},\qquad
\omega(f_i,e_j)=-\delta_{ij},\qquad
\omega(e_i,e_j)=\omega(f_i,f_j)=0.
\]

For a pure tensor
$v=v_1\otimes\cdots\otimes v_{2m}\in W^{\otimes 2m},$
define the \(2m\times 2m\) matrix \(B_G(v)\) by
\[
(B_G(v))_{ij}
:=
-L_{ij}\omega(v_i,v_j).
\]
Since \(L\) is symmetric and \(\omega\) is skew-symmetric, \(B_G(v)\) is
skew-symmetric. We then set
\[
Z_G(v_1\otimes\cdots\otimes v_{2m})
:=
\operatorname{Pf}(B_G(v)).
\]
The Pfaffian expression is multilinear in \(v_1,\ldots,v_{2m}\), and
therefore extends uniquely to a linear functional
\[
Z_G\in
\operatorname{Hom}\!\left(W^{\otimes 2m},\mathbb C\right).
\]

Thus, for any \(g\in \Sp(2n,\mathbb C)\),
\[
(B_{G}(gv))_{ij}
=
-L_{ij}\,\omega(gv_i,gv_j)
=
-L_{ij}\,\omega(v_i,v_j)
=
(B_{G}(v))_{ij}.
\]
It follows that
\[
 Z_{G}(gv_1\otimes\dots\otimes gv_{2m})
=
 Z_{G}(v_1\otimes \dots\otimes v_{2m}),
\]
so that
\(
 Z_{G}
\in
\Hom_{\Sp(2n,\mathbb C)}
\left(W^{\otimes 2m},\mathbb C\right).
\)

\subsection{Pairing tensors}

Let $\mathcal M(2m)$ be the set of perfect matchings of $N$.
Draw a matching by chords in a disk with boundary vertices in their
specified cyclic order. Two pairs $(a,b)$ and $(c,d)$ cross if their endpoints alternate.
Write $\mathrm{cr}(\pi)$ for the number of unordered pairs of crossing chords.

\begin{defn}\label{def:pairing_tensors}
For \(\pi\in\mathcal M(2m)\), define
\(
T_\pi\in
\Hom_{\Sp(2n)}\!\left(W^{\otimes 2m},\mathbb C\right)
\)
on pure tensors by
\[
T_\pi(v_1\otimes\cdots\otimes v_{2m})
:=
\prod_{\substack{\{i,j\}\in\pi\\ i<j}}
\omega(v_i,v_j),
\]
and extend linearly to \(W^{\otimes 2m}\).
\end{defn}

\begin{lem} \label{lem:pairing_basis}
We have
\[
Z_{G}
=
\sum_{\pi\in M(2m)}
(-1)^{\operatorname{cr}(\pi)+m}
\left(
\prod_{\{i,j\}\in\pi} L_{ij}
\right)
T_\pi.
\]
Moreover, if $\pi\in M(2m)$ is regarded as a network whose edges
have conductance $1$, possibly nonplanar, then
\(
Z_{\pi}=(-1)^{\operatorname{cr}(\pi)}T_\pi.
\)
\end{lem}

\begin{proof}
Recall that,
\(
(B_G(v))_{ij}=-L_{ij}\omega(v_i,v_j).
\)
Expanding the Pfaffian over perfect matchings gives
\[
Z_{G}(v)
=
\operatorname{Pf}(B_G(v))
=
\sum_{\pi\in M(2m)}
\operatorname{sgn}(\pi)
\prod_{\{i,j\}\in\pi,\;i<j}
\bigl(-L_{ij}\omega(v_i,v_j)\bigr).
\]
With the usual Pfaffian convention, the sign of a matching is
determined by the parity of the number of crossings of its chords.
Taking into account the minus sign in each entry of $B_G(v)$ gives
\[
Z_{G}(v)
=
\sum_{\pi\in M(2m)}
(-1)^{\operatorname{cr}(\pi)+m}
\left(
\prod_{\{i,j\}\in\pi}L_{ij}
\right)
T_\pi(v).
\]

Now regard $\pi$ itself as a network with conductance $1$ on each
edge.  For $i\neq j$,
\[
L_{ij}=
\begin{cases}
-1, & \{i,j\}\in\pi,\\
0, & \text{otherwise}.
\end{cases}
\]
Hence only the term corresponding to $\pi$ contributes to the Pfaffian expansion.
  Therefore
\(
Z_{\pi}=(-1)^{\operatorname{cr}(\pi)}T_\pi.
\)
\end{proof}

We recall a special case of a theorem of Kenyon and Wilson
\cite{kenyon2008combinatorics}. Our response matrix is the negative of the
one used in \cite{kenyon2008combinatorics}, accounting for the factor
$(-1)^k$.

\begin{thm} \label{thm:grove_expansion}
Let $N=Q\sqcup R\sqcup S$, where
$R=\{r_1<\cdots<r_k\}$ and $S=\{s_1<\cdots<s_k\}$.
Write $L_R^S$ for the submatrix with rows indexed by $R$
and columns indexed by $S$, both in increasing order. Then
\begin{align} \label{eq:KW_formula}
\det L_R^S
=
\frac{(-1)^k}{Z_{\mathrm{unc}}}
\sum_{\rho\in S_k}\operatorname{sgn}(\rho)\,
Z[r_1,s_{\rho(1)}\mid\cdots\mid r_k,s_{\rho(k)}
  \mid q_1\mid\cdots\mid q_\ell].
\end{align}
Here $Q=\{q_1,\ldots,q_\ell\}$, and each $q_i$ is a
singleton boundary block.
\end{thm}

\begin{defn}\label{defn:invariant_assigned}
Let $H$ have vertex set $V(G)$ and edge multiplicities
$\mu_H:E(G)\to\{0,\ldots,n\}$.

An ordered partition
$\mathscr P=(R_1,S_1,\ldots,R_n,S_n)$ of $N$ is
\emph{balanced} if $|R_\ell|=|S_\ell|=:k_\ell$ for every $\ell$.
For such a partition, define
\[
v_{\mathscr P}:=\bigotimes_{i=1}^{2m}v_i,
\qquad
v_i=
\begin{cases}
e_\ell, & i\in R_\ell,\\
f_\ell, & i\in S_\ell.
\end{cases}
\]

An \emph{admissible coloring} of $H$ compatible with
$\mathscr P$ is an ordered tuple of spanning groves
$(F_1,\ldots,F_n)$ of $G$ such that each edge $e$ belongs
to exactly $\mu_H(e)$ of the groves, and each component
of $F_\ell$ contains either:
\begin{itemize}
\item exactly two boundary vertices, one in $R_\ell$
and one in $S_\ell$; or
\item exactly one boundary vertex, belonging to
$N\setminus(R_\ell\cup S_\ell)$.
\end{itemize}
Write $\operatorname{Col}(H;\mathscr P)$ for the set of
these colorings.

Let $\sigma_{\mathscr P}\in S_{2m}$ be the permutation
such that
$(\sigma_{\mathscr P}(1),\ldots,\sigma_{\mathscr P}(2m))$
lists the elements of $R_1,S_1,\ldots,R_n,S_n$ in that
order, increasingly within each set. Define
\[
\epsilon_{\mathscr P}
:=
\operatorname{sgn}(\sigma_{\mathscr P})
(-1)^{\sum_{\ell=1}^n\binom{k_\ell}{2}}.
\]
For $c\in\operatorname{Col}(H;\mathscr P)$, write
$R_\ell=\{a_{\ell,1}<\cdots<a_{\ell,k_\ell}\}$ and
$S_\ell=\{b_{\ell,1}<\cdots<b_{\ell,k_\ell}\}$.
Let $\rho_\ell(c)\in S_{k_\ell}$ be determined by requiring
$a_{\ell,j}$ and $b_{\ell,\rho_\ell(c)(j)}$ to lie in the same
component of $F_\ell$. Set
\[
\operatorname{sgn}_{\mathscr P}(c)
:=
\epsilon_{\mathscr P}
\prod_{\ell=1}^n\operatorname{sgn}(\rho_\ell(c)),
\qquad
T_H(v_{\mathscr P})
:=
\sum_{c\in\operatorname{Col}(H;\mathscr P)}
\operatorname{sgn}_{\mathscr P}(c).
\]
Set $T_H$ equal to zero on unbalanced basis tensors
and extend linearly to obtain
\[
T_H\in\operatorname{Hom}(W^{\otimes 2m},\mathbb C).
\]
\end{defn}

\begin{remark}
Definition~\ref{defn:invariant_assigned} depends only on $H$,
so its vertex and edge sets do not need to agree with those of $G$.
It also applies to nonplanar graphs drawn in the disk with
node set $N$ on the boundary.
\end{remark}

\begin{figure}[ht]
\centering
\begin{tikzpicture}[
    scale=.9,
    every node/.style={font=\small},
    redlayer/.style={red!75!black,line width=1pt},
    bluelayer/.style={
    blue!75!black,
    line width=1pt,
    dash pattern=on 2pt off 1pt
}
]
\foreach \panel/\xx/\yy in {
    0/0/0,
    1/5.1/2.85,
    2/10.2/2.85,
    3/5.1/-2.85,
    4/10.2/-2.85
}{
\begin{scope}[shift={(\xx,\yy)}]
    \draw[black!35,line width=.4pt]
        (0,0) circle ({sqrt(2.32)});

    \coordinate (A) at (-.5,.5);
    \coordinate (B) at (.5,.5);
    \coordinate (C) at (.5,-.5);
    \coordinate (D) at (-.5,-.5);
    \coordinate (O) at (0,0);
    \coordinate (Q) at (-.85,-.85);
    \coordinate (P) at (.95,1.0);
    \coordinate (N3) at (.6,1.4);

    \ifnum\panel=0
        \draw[line width=.85pt]
            (-1.4,.6)--(A)--(-.6,1.4)
            (N3)--(B)--(1.4,.6)
            (1.4,-.6)--(C)--(.6,-1.4)
            (-.6,-1.4)--(Q)--(-1.4,-.6)
            (A)--(B)--(C)--(D)--cycle
            (A)--(O)--(C);

        \draw[line width=.85pt]
            (Q) to[bend left=18] (D)
            (Q) to[bend right=18] (D)
            (N3) to[bend left=18] (P)
            (N3) to[bend right=18] (P);
    \else
        \draw[redlayer]
            (-1.4,.6)--(A)
            (N3)--(B)
            (1.4,-.6)--(C)
            (-.6,-1.4)--(Q);
        \draw[bluelayer]
            (-.6,1.4)--(A)
            (1.4,.6)--(B)
            (.6,-1.4)--(C)
            (-1.4,-.6)--(Q);

        \draw[redlayer]
            (Q) to[bend left=18] (D)
            (N3) to[bend left=18] (P);
        \draw[bluelayer]
            (Q) to[bend right=18] (D)
            (N3) to[bend right=18] (P);

        \ifnum\panel<3
            \draw[redlayer] (A)--(B) (D)--(C);
            \draw[bluelayer] (A)--(D) (B)--(C);
        \else
            \draw[bluelayer] (A)--(B) (D)--(C);
            \draw[redlayer] (A)--(D) (B)--(C);
        \fi

        \ifodd\panel
            \draw[redlayer] (A)--(O);
            \draw[bluelayer] (O)--(C);
        \else
            \draw[bluelayer] (A)--(O);
            \draw[redlayer] (O)--(C);
        \fi
    \fi

    \foreach \p in {A,B,C,D,O,Q,P}{
        \fill (\p) circle (1.8pt);
    }

    \foreach \x/\y in {
        -1.4/.6,-.6/1.4,.6/1.4,1.4/.6,
        1.4/-.6,.6/-1.4,-.6/-1.4,-1.4/-.6
    }{
        \fill (\x,\y) circle (1.8pt);
    }

    \ifnum\panel=0
        \node[left=3pt]  at (-1.4,.6) {$1$};
        \node[above=3pt] at (-.6,1.4) {$2$};
        \node[above=3pt] at (.6,1.4) {$3$};
        \node[right=3pt] at (1.4,.6) {$4$};
        \node[right=3pt] at (1.4,-.6) {$5$};
        \node[below=3pt] at (.6,-1.4) {$6$};
        \node[below=3pt] at (-.6,-1.4) {$7$};
        \node[left=3pt]  at (-1.4,-.6) {$8$};
        \node at (0,-2.65) {The superposition $H$};
    \else
        \node[left=3pt]  at (-1.4,.6) {$e_1$};
        \node[above=3pt] at (-.6,1.4) {$e_2$};
        \node[above=3pt] at (.6,1.4) {$f_1$};
        \node[right=3pt] at (1.4,.6) {$f_2$};
        \node[right=3pt] at (1.4,-.6) {$e_1$};
        \node[below=3pt] at (.6,-1.4) {$e_2$};
        \node[below=3pt] at (-.6,-1.4) {$f_1$};
        \node[left=3pt]  at (-1.4,-.6) {$f_2$};
        \node at (0,-2.65) {Coloring $c_{\panel}$};
    \fi
\end{scope}
}
\end{tikzpicture}
\caption{A double-grove superposition $H$ and its four
compatible colorings for
$\mathscr P=(\{1,5\},\{3,7\},\{2,6\},\{4,8\})$.
Solid red and dashed blue edges denote the first and
second grove layers, respectively. Curved edges represent doubled edges. Each doubled edge
contributes one edge to each layer.
Both layers span all vertices. A singleton boundary
block may be empty or contain interior vertices, as illustrated by the component at node $3$.}
\label{fig:two-colorings}
\end{figure}
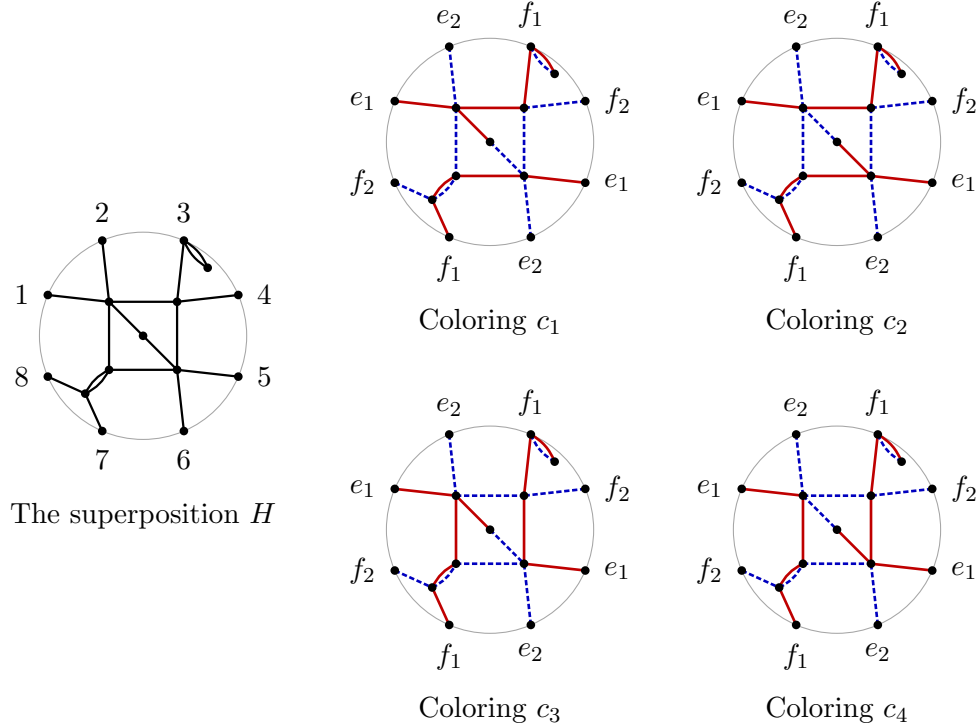

\begin{example}
Suppose $n=2$. For the graph $H$ in
Figure~\ref{fig:two-colorings}, let
\[
v=e_1\otimes e_2\otimes f_1\otimes f_2
  \otimes e_1\otimes e_2\otimes f_1\otimes f_2.
\]
Each doubled edge must contribute one edge to each grove
layer. Contracting these doubled edges leaves the square
with its diagonal and the eight boundary edges. The colors of the boundary edges
are determined by $v$.

The two halves of the diagonal must have opposite
colors. If both were red or blue, then the vertex in the middle of the square would not be in a component of the opposite color.
The square must alternate colors, since
two adjacent square edges of the same color would
connect three boundary nodes in that layer.
Thus there are two choices for the square and two
for the diagonal, giving exactly the four compatible
colorings shown.

The boundary partitions of the two layers are
\[
\begin{array}{c|c|c}
 & \text{Red layer} & \text{Blue layer} \\ \hline
c_1,c_2 & 13|57|2|4|6|8 & 28|46|1|3|5|7 \\
c_3,c_4 & 17|35|2|4|6|8 & 24|68|1|3|5|7
\end{array}
\]

Here $\epsilon_{\mathscr P}=1$. Each coloring has sign
$-1$ since for each coloring one layer gives the
identity pairing between its ordered sets $R_i$ and
$S_i$, while the other gives a transposition.
Therefore
$T_H(v)=-4.$
\end{example}

The following evaluation formula relates the coloring
definition to the network tensor.

\begin{lem}\label{lem:sp4_test_evaluation}
For every ordered balanced partition
$\mathscr P=(R_1,S_1,\ldots,R_n,S_n)$ of $N$,
\begin{align} \label{eq:L_expansion}
Z_G(v_{\mathscr P})
=
(-1)^m\epsilon_{\mathscr P}
\prod_{\ell=1}^n \det L_{R_\ell}^{S_\ell},
\end{align}
where $L_{R_\ell}^{S_\ell}$ has rows indexed by $R_\ell$ and
columns indexed by $S_\ell$, both in increasing order.
\end{lem}

\begin{proof}
Write $k_\ell=|R_\ell|=|S_\ell|$, so that $\sum_{\ell=1}^n k_\ell=m$.
Since
\[
(B_G(v_{\mathscr P}))_{ij}=-L_{ij}\omega(v_i,v_j),
\]
reordering the rows and columns by $\sigma_{\mathscr P}$
gives the block diagonal matrix
\[
\bigoplus_{\ell=1}^n
\begin{pmatrix}
0 & -L_{R_\ell}^{S_\ell}\\
(L_{R_\ell}^{S_\ell})^{\mathsf T} & 0
\end{pmatrix}.
\]
Using the identity
\[
\operatorname{Pf}
\begin{pmatrix}
0 & -A\\
A^{\mathsf T} & 0
\end{pmatrix}
=
(-1)^{\binom k2+k}\det A
\qquad
(A\in\operatorname{Mat}_{k\times k}),
\]
we obtain
\[
\begin{aligned}
Z_G(v_{\mathscr P})
&=
\operatorname{sgn}(\sigma_{\mathscr P})
(-1)^{\sum_{\ell=1}^n(\binom{k_\ell}{2}+k_\ell)}
\prod_{\ell=1}^n \det L_{R_\ell}^{S_\ell}\\
&=
(-1)^m\epsilon_{\mathscr P}
\prod_{\ell=1}^n \det L_{R_\ell}^{S_\ell}.
\end{aligned}
\]
Here empty blocks contribute $1$.
\end{proof}

\begin{lem} \label{lem:grove_sp4}
For every $H$,
\(
T_H\in\operatorname{Hom}_{\mathrm{Sp}(2n)}
(W^{\otimes 2m},\mathbb C).
\)
\end{lem}

\begin{proof}
Assign a variable conductance $x_e$ to each edge
$e\in E(G)$, and define
\[
P_G(\mathbf x;v)
:=
Z_{\mathrm{unc},G}(\mathbf x)^n
Z_G(\mathbf x;v).
\]
For a balanced basis tensor $v=v_{\mathscr P}$,
Lemma~\ref{lem:sp4_test_evaluation} and the grove
Theorem \ref{thm:grove_expansion} give
\[
P_G(\mathbf x;v_{\mathscr P})
=
\sum_{(F_1,\ldots,F_n)}
\epsilon_{\mathscr P}
\prod_{\ell=1}^n
\left(
\operatorname{sgn}(\rho_\ell)
\prod_{e\in F_\ell}x_e
\right),
\]
where the sum is over ordered $n$-tuples of groves of $G$
compatible with $\mathscr P$. The signs from \eqref{eq:KW_formula}
formulas multiply to $(-1)^{\sum_\ell k_\ell}=(-1)^m$,
canceling the factor $(-1)^m$ in
Lemma~\ref{lem:sp4_test_evaluation}.

The coefficient of
$\prod_{e\in E(G)}x_e^{\mu_H(e)}$ selects precisely those
tuples in which each edge $e$ belongs to exactly
$\mu_H(e)$ layers. These are the admissible colorings
of $H$. Therefore,
\[
T_H(v_{\mathscr P})
=
\left[\prod_{e\in E(G)}x_e^{\mu_H(e)}\right]
P_G(\mathbf x;v_{\mathscr P}).
\]
Both sides vanish on unbalanced basis tensors.
Thus, by linearity, $P_G(\mathbf x;v)$ is polynomial
in $\mathbf x$ and
\[
T_H(v)
=
\left[\prod_{e\in E(G)}x_e^{\mu_H(e)}\right]
P_G(\mathbf x;v)
\]
for every $v\in W^{\otimes 2m}$.

For $g\in\Sp(2n)$, the Pfaffian definition gives
\[
P_G(\mathbf x;g^{\otimes 2m}v)
=
P_G(\mathbf x;v).
\]
Taking the coefficient of
$\prod_{e\in E(G)}x_e^{\mu_H(e)}$ on both sides yields
\(
T_H(g^{\otimes 2m}v)=T_H(v),
\)
as required.
\end{proof}

\section{$\SL(2,\mathbb C)$ case}

Let $\mathrm{NC}(2m)$ denote the set of non crossing matchings of $N$ that can be drawn by connecting paired vertices in the disk without crossings. For $\pi\in\mathrm{NC}(2m)$, regard $\pi$ as a planar network
with unit conductances and no interior vertices. By Lemma~\ref{lem:pairing_basis},
\(
Z_\pi=T_\pi.
\)
Thus $T_\pi$ is the usual Temperley--Lieb invariant associated
with $\pi$.

\begin{prop}\label{prop:sl2_grove_expansion}
Suppose $n=1$. For
$\pi\in\mathrm{NC}(2m)$, let $Z[\pi]$ denote the weighted sum of groves
whose boundary connectivity is $\pi$. Then
\[
Z_{G}
=
\frac{1}{Z_{\mathrm{unc}}}
\sum_{\pi\in\mathrm{NC}(2m)}
Z[\pi]\,T_\pi.
\]
Equivalently, for
$v=v_1\otimes\cdots\otimes v_{2m}\in W^{\otimes 2m}$,
\[
Z_{G}(v)
=
\frac{1}{Z_{\mathrm{unc}}}
\sum_{\pi\in\mathrm{NC}(2m)}
Z[\pi]\,
\Pf(B_\pi(v)).
\]
\end{prop}

\begin{proof}
It suffices to
evaluate both sides on tensors whose factors are the symplectic basis elements
$\{e_1,f_1\}$.

Fix such a tensor $v$, and write
\(
N=R\sqcup S,
\)
where the vertices in $R$ are assigned $e_1$ and those in $S$ are
assigned $f_1$. If $|R|\neq |S|$, then both sides vanish. We therefore
assume that $|R|=|S|=m$.

The only nonzero pairings among the boundary vectors are
$\omega(e_1,f_1)=1$ and $\omega(f_1,e_1)=-1.$
Let $\sigma_{(R,S)}$ denote the sign of the permutation that changes the
boundary order $1,\ldots,2m$ to
\(
r_1,\ldots,r_m,s_1,\ldots,s_m.
\)
In this ordering, the matrix $B_G(v)$ has the block form
\[
\begin{pmatrix}
0 & -L_R^S\\
(L_R^S)^{\mathsf T} & 0
\end{pmatrix}.
\]
Hence
\(
Z_{G}(v)
=\Pf(B_G(v))
=\sigma_{(R,S)}(-1)^{\binom m2+m}\det L_R^S.
\) The grove Theorem \ref{thm:grove_expansion} gives that
\[
Z_{G}(v)
=
\frac{\sigma_{(R,S)}(-1)^{{m\choose 2}}}{Z_{\mathrm{unc}}}
\sum_{\pi}
\sgn(\rho_\pi)Z[\pi],
\]
where the sum ranges over noncrossing pairings $\pi$ in which every
pair contains one vertex of $R$ and one vertex of $S$, and
$\sgn(\rho_\pi)$ is the sign of the permutation appearing in the RHS of \eqref{eq:KW_formula}.

It remains only to identify the sign with $T_\pi(v)$. For this, regard $\pi$ itself
as a planar network, with each edge having conductance one. Its
boundary response tensor is precisely $T_\pi$. Moreover, there is a
unique grove on this network having boundary connectivity $\pi$, and
its weight is one. Applying the same Theorem \ref{thm:grove_expansion} to the
network $\pi$ with edge conductances $c_e=1\, \forall e\in \pi$ gives
\[
T_\pi(v)=\sigma_{(R,S)}(-1)^{{m\choose 2}}\sgn(\rho_\pi).
\]
Since there is only one valid grove on the graph $\pi$ and $Z_{\mathrm{unc}}=1$, we obtain
\[
Z_{G}(v)
=
\frac{1}{Z_{\mathrm{unc}}}
\sum_{\pi\in\mathrm{NC}(2m)}Z[\pi]T_\pi(v).
\]
Since the tensors with factors in $\{e_1,f_1\}$ form a basis of
$W^{\otimes 2m}$, the claimed identity follows.
\end{proof}

\section{The case $\Sp(4,\mathbb C)$}
\label{sec:sec_4}

We now specialize to $n=2$, so that $W=\mathbb C^4$.
Our goal is to express $Z_G$ in a basis of
$\Sp(4)$-invariants, with coefficients given by partition
functions of reduced double groves.

For an ordered balanced partition
$\mathscr P=(R_1,S_1,R_2,S_2)$ of $N$, call an ordered pair
$(F_1,F_2)$ of groves compatible with $\mathscr P$ a
\emph{double grove}. Thus each component of $F_\ell$ contains
either two boundary vertices, one in $R_\ell$ and one in $S_\ell$,
or a single boundary vertex in $N\setminus(R_\ell\cup S_\ell)$.
The superposition $H=F_1\cup F_2$ records each edge with
its multiplicity in the two layers.

By Lemma~\ref{lem:sp4_test_evaluation} and
Theorem~\ref{thm:grove_expansion},
\[
Z_{\mathrm{unc}}^2Z_G(v_{\mathscr P})
=
\sum_{(F_1,F_2)}
\epsilon_{\mathscr P}
\operatorname{sgn}(\rho_1)\operatorname{sgn}(\rho_2)
\prod_{e\in F_1}c_e\prod_{e\in F_2}c_e,
\]
where the sum is over double groves compatible with
$\mathscr P$. We will reduce their superpositions to
$3$-noncrossing matching graphs and show that the
corresponding coloring tensors form a basis.

\begin{defn}
A \emph{$3$-noncrossing matching graph} on $2m$ boundary vertices is
obtained from a chord diagram having no three mutually crossing chords
by replacing every crossing of two chords by a $4$-valent vertex.
The resulting graph has $1$-valent boundary vertices and $4$-valent
interior vertices.
\end{defn}

For background on $k$-noncrossing matchings and their
enumeration, see \cite{chen2007crossings}. Given a double grove $(F_1,F_2)$, we apply a slight modification of the double grove
reduction of \cite{gao2025electrical} to its superposition
\(H=F_1\cup F_2,
\)
with edges counted with multiplicity. Because the boundary blocks in
each grove are either pairs or singletons, terms in the reduction algorithm that identify
boundary vertices do not contribute. We therefore use the following
boundary-adapted form of the reduction rules.

Let $\mathrm{TC}(2m)$ denote the set of $3$-noncrossing matching graphs
on the boundary node set $N=\{1,\ldots,2m\}$, and let
$\mathbb Z\mathrm{TC}(2m)$ denote the free abelian group with basis
$\mathrm{TC}(2m)$.

Following \cite{gao2025electrical}, for a double grove $H$ we define its
grove reduction
\(
\alpha(H)\in\mathbb Z\mathrm{TC}(2m)
\)
recursively by the following rules. For each $H$, choose a complete reduction sequence and let
$\alpha(H)$ denote its output. Note that in \cite{gao2025electrical}, the \(3\)-noncrossing matching is associated to the medial pairing rather than directly on the primal graph.

Independence of this choice
will follow from the basis Corollary \ref{cor:sp4_basis}.

\begin{defn}(cf.~\cite[Definition~4.2]{gao2025electrical})\label{defn:groves_moves}
For a double-grove superposition $H$, with edges counted with multiplicity,
define $\alpha(H)$ recursively as follows:

\begin{enumerate}
\item[(0)] If some edge of \(H\) has multiplicity \(\ge 3\), then
\(
\alpha(H)=0.
\)

\item[(1)] If $H$ is a $3$-noncrossing matching graph $\xi$, then $\alpha(H)=\xi$.

\item[(2)] If \(H\) has an isolated boundary vertex or if \(H\) has loops, then
\(
\alpha(H)=0.
\)

\item[(3)] If $H$ has a double edge $e$, then
\[
\alpha(H)=
\begin{cases}
\alpha(H'), & \text{if the endpoints of $e$ are not both boundary vertices},\\
0, & \text{if both endpoints of $e$ are boundary vertices},
\end{cases}
\]
where $H'$ is obtained from $H$ by contracting $e$.

\item[(4)] Suppose that $H$ has edges $e_1,\ldots,e_r$ of multiplicity $1$
between the same pair of vertices with $r\geq 2$. If both vertices are boundary vertices, then
\(
\alpha(H)=0.
\)
Otherwise,
\(
\alpha(H)=0 \text{ if } r\ge 3,
\)
while, if $r=2$,
\(
\alpha(H)=2\,\alpha(H'),
\)
where $H'$ is obtained from $H$ by contracting $e_1$ and removing $e_2$.

\item[(5)] If \(H\) has an interior leaf or an interior vertex with
degree \(0\), then
\(
\alpha(H)=0.
\)

\item[(6)] If \(H\) has an interior vertex \(v\) of degree \(2\) with
distinct neighbors, then
\(
\alpha(H)=2\,\alpha(H'),
\)
where \(H'\) is obtained from \(H\) by removing \(v\) and the two edges.

\item[(7)] If \(H\) has an interior vertex \(v\) of degree \(3\) with
distinct neighbors \(v_1,v_2,v_3\), then
\(
\alpha(H)=\alpha(H_1)+\alpha(H_2)+\alpha(H_3),
\)
where \(H_i\) is obtained from \(H\) by removing \(v\) together with the
edges incident to it and adding the edge \((v_{i+1},v_{i+2})\), where
the indices are taken modulo \(3\).

\item[(8)] Suppose that $H$ has a triangular face with vertices
$v_1,v_2,v_3$. For $i\in\{1,2,3\}$, let $H_i$ be obtained from $H$
by identifying $v_i$ with $v_{i+1}$ and replacing the three edges of
the triangle by a single edge between $v_i=v_{i+1}$ and $v_{i+2}$,
with indices taken modulo $3$.

If at most one of $v_1,v_2,v_3$ is a boundary vertex, then
\[
\alpha(H)=\alpha(H_1)+\alpha(H_2)+\alpha(H_3).
\]

If exactly two of $v_1,v_2,v_3$ are boundary vertices, say $v_1$ and
$v_2$, then
\[
\alpha(H)=\alpha(H_2)+\alpha(H_3),
\]
so that the term $H_1$, which identifies the two boundary vertices,
is omitted. (See Figure \ref{fig:boundary_triangle_colored})

If all three of $v_1,v_2,v_3$ are boundary vertices, then
\(
\alpha(H)=0.
\)

\end{enumerate}
See figure \ref{fig:grove-reduction-moves}.
\end{defn}

Every nonterminal reduction either gives zero or replaces $H$
with double groves that have fewer vertices or edges. Hence, every reduction sequence
terminates. We say a double grove is \emph{reduced} if there are no further reductions that can be applied to it.

Proposition \ref{prop:reduction_3_matching} below shows that every nonzero
reduced double grove is a 
$3$-noncrossing matching graph.

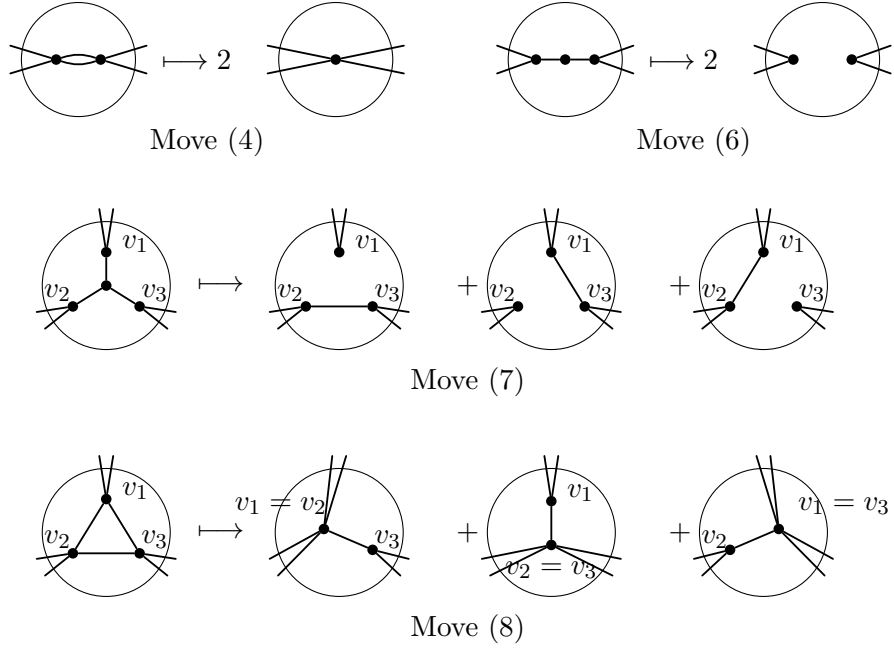
\begin{figure}[ht]
\centering
\begin{tikzpicture}[
    scale=0.92,
    line cap=round,
    line join=round,
    vertex/.style={circle,fill=black,inner sep=1.4pt},
    edge/.style={line width=0.7pt},
    every node/.style={font=\small}
]


\begin{scope}[shift={(-0.4,0)}]

\draw (0,0) circle (0.82);
\node[vertex] (a) at (-0.32,0) {};
\node[vertex] (b) at ( 0.32,0) {};

\draw[edge] (a) to[bend left=18] (b);
\draw[edge] (a) to[bend right=18] (b);

\draw[edge] (a)--(-0.98, 0.20);
\draw[edge] (a)--(-0.98,-0.20);
\draw[edge] (b)--( 0.98, 0.20);
\draw[edge] (b)--( 0.98,-0.20);

\node at (1.70,0) {$\longmapsto 2$};

\begin{scope}[shift={(3.70,0)}]
\draw (0,0) circle (0.82);
\node[vertex] (c) at (0,0) {};

\draw[edge] (c)--(-0.98, 0.20);
\draw[edge] (c)--(-0.98,-0.20);
\draw[edge] (c)--( 0.98, 0.20);
\draw[edge] (c)--( 0.98,-0.20);
\end{scope}

\node at (1.85,-1.18) {Move (4)};
\end{scope}


\begin{scope}[shift={(6.6,0)}]

\draw (0,0) circle (0.82);
\node[vertex] (a) at (-0.42,0) {};
\node[vertex] (v) at ( 0,0) {};
\node[vertex] (b) at ( 0.42,0) {};

\draw[edge] (a)--(v)--(b);

\draw[edge] (a)--(-0.98, 0.20);
\draw[edge] (a)--(-0.98,-0.20);
\draw[edge] (b)--( 0.98, 0.20);
\draw[edge] (b)--( 0.98,-0.20);

\node at (1.70,0) {$\longmapsto 2$};

\begin{scope}[shift={(3.70,0)}]
\draw (0,0) circle (0.82);
\node[vertex] (a2) at (-0.42,0) {};
\node[vertex] (b2) at ( 0.42,0) {};


\draw[edge] (a2)--(-0.98, 0.20);
\draw[edge] (a2)--(-0.98,-0.20);
\draw[edge] (b2)--( 0.98, 0.20);
\draw[edge] (b2)--( 0.98,-0.20);
\end{scope}

\node at (1.85,-1.18) {Move (6)};
\end{scope}


\begin{scope}[shift={(0,-3.25)}]

\draw (0,0) circle (0.92);

\node[vertex] (v)  at (0,0) {};
\node[vertex] (v1) at (0,0.48) {};
\node[vertex] (v2) at (-0.48,-0.30) {};
\node[vertex] (v3) at (0.48,-0.30) {};

\draw[edge] (v)--(v1);
\draw[edge] (v)--(v2);
\draw[edge] (v)--(v3);

\draw[edge] (v1)--(-0.10,1.10);
\draw[edge] (v1)--( 0.10,1.10);

\draw[edge] (v2)--(-1.00,-0.38);
\draw[edge] (v2)--(-0.88,-0.62);

\draw[edge] (v3)--( 1.00,-0.38);
\draw[edge] (v3)--( 0.88,-0.62);

\node[anchor=west] at (0.07,0.60) {$v_1$};
\node[anchor=east] at (-0.35,-0.1) {$v_2$};
\node[anchor=west] at (0.35,-0.1) {$v_3$};

\node at (1.65,0) {$\longmapsto$};

\begin{scope}[shift={(3.35,0)}]
\draw (0,0) circle (0.92);

\node[vertex] (a1) at (0,0.48) {};
\node[vertex] (a2) at (-0.48,-0.30) {};
\node[vertex] (a3) at (0.48,-0.30) {};

\draw[edge] (a2)--(a3);

\draw[edge] (a1)--(-0.10,1.10);
\draw[edge] (a1)--( 0.10,1.10);

\draw[edge] (a2)--(-1.00,-0.38);
\draw[edge] (a2)--(-0.88,-0.62);

\draw[edge] (a3)--( 1.00,-0.38);
\draw[edge] (a3)--( 0.88,-0.62);

\node[anchor=west] at (0.07,0.60) {$v_1$};
\node[anchor=east] at (-0.35,-0.1) {$v_2$};
\node[anchor=west] at (0.35,-0.1) {$v_3$};
\end{scope}

\node at (5.20,0) {$+$};

\begin{scope}[shift={(6.40,0)}]
\draw (0,0) circle (0.92);

\node[vertex] (b1) at (0,0.48) {};
\node[vertex] (b2) at (-0.48,-0.30) {};
\node[vertex] (b3) at (0.48,-0.30) {};

\draw[edge] (b1)--(b3);

\draw[edge] (b1)--(-0.10,1.10);
\draw[edge] (b1)--( 0.10,1.10);

\draw[edge] (b2)--(-1.00,-0.38);
\draw[edge] (b2)--(-0.88,-0.62);

\draw[edge] (b3)--( 1.00,-0.38);
\draw[edge] (b3)--( 0.88,-0.62);

\node[anchor=west] at (0.07,0.60) {$v_1$};
\node[anchor=east] at (-0.35,-0.1) {$v_2$};
\node[anchor=west] at (0.35,-0.1) {$v_3$};
\end{scope}

\node at (8.25,0) {$+$};

\begin{scope}[shift={(9.45,0)}]
\draw (0,0) circle (0.92);

\node[vertex] (c1) at (0,0.48) {};
\node[vertex] (c2) at (-0.48,-0.30) {};
\node[vertex] (c3) at (0.48,-0.30) {};

\draw[edge] (c1)--(c2);

\draw[edge] (c1)--(-0.10,1.10);
\draw[edge] (c1)--( 0.10,1.10);

\draw[edge] (c2)--(-1.00,-0.38);
\draw[edge] (c2)--(-0.88,-0.62);

\draw[edge] (c3)--( 1.00,-0.38);
\draw[edge] (c3)--( 0.88,-0.62);

\node[anchor=west] at (0.07,0.60) {$v_1$};
\node[anchor=east] at (-0.35,-0.1) {$v_2$};
\node[anchor=west] at (0.35,-0.1) {$v_3$};
\end{scope}

\node at (5.20,-1.38) {Move (7)};
\end{scope}


\begin{scope}[shift={(0,-6.8)}]

\draw (0,0) circle (0.92);

\node[vertex] (v1) at (0,0.48) {};
\node[vertex] (v2) at (-0.48,-0.30) {};
\node[vertex] (v3) at (0.48,-0.30) {};

\draw[edge] (v1)--(v2)--(v3)--(v1)--cycle;

\draw[edge] (v1)--(-0.10,1.10);
\draw[edge] (v1)--( 0.10,1.10);

\draw[edge] (v2)--(-1.00,-0.38);
\draw[edge] (v2)--(-0.88,-0.62);

\draw[edge] (v3)--( 1.00,-0.38);
\draw[edge] (v3)--( 0.88,-0.62);

\node[anchor=west] at (0.07,0.60) {$v_1$};
\node[anchor=east] at (-0.35,-0.1) {$v_2$};
\node[anchor=west] at (0.35,-0.1) {$v_3$};

\node at (1.65,0) {$\longmapsto$};

\begin{scope}[shift={(3.35,0)}]
\draw (0,0) circle (0.92);

\node[vertex] (a12) at (-0.22,0.05) {};
\node[vertex] (a3)  at ( 0.48,-0.25) {};

\draw[edge] (a12)--(a3);

\draw[edge] (a12)--(-0.10,1.10);
\draw[edge] (a12)--( 0.10,1.10);
\draw[edge] (a12)--(-1.00,-0.38);
\draw[edge] (a12)--(-0.88,-0.62);

\draw[edge] (a3)--( 1.00,-0.38);
\draw[edge] (a3)--( 0.88,-0.62);

\node[anchor=east] at (-0.02,0.38) {$v_1=v_2$};
\node[anchor=west] at (0.35,-0.1) {$v_3$};
\end{scope}

\node at (5.20,0) {$+$};

\begin{scope}[shift={(6.40,0)}]
\draw (0,0) circle (0.92);

\node[vertex] (b1)  at (0,0.45) {};
\node[vertex] (b23) at (0,-0.18) {};

\draw[edge] (b1)--(b23);

\draw[edge] (b1)--(-0.10,1.10);
\draw[edge] (b1)--( 0.10,1.10);

\draw[edge] (b23)--(-1.00,-0.38);
\draw[edge] (b23)--(-0.88,-0.62);
\draw[edge] (b23)--( 1.00,-0.38);
\draw[edge] (b23)--( 0.88,-0.62);

\node[anchor=west] at (0.07,0.57) {$v_1$};
\node[anchor=north] at (0,-0.26) {$v_2=v_3$};
\end{scope}

\node at (8.25,0) {$+$};

\begin{scope}[shift={(9.45,0)}]
\draw (0,0) circle (0.92);

\node[vertex] (c2)  at (-0.48,-0.25) {};
\node[vertex] (c13) at ( 0.22,0.05) {};

\draw[edge] (c2)--(c13);

\draw[edge] (c2)--(-1.00,-0.38);
\draw[edge] (c2)--(-0.88,-0.62);

\draw[edge] (c13)--(-0.10,1.10);
\draw[edge] (c13)--( 0.10,1.10);
\draw[edge] (c13)--( 1.00,-0.38);
\draw[edge] (c13)--( 0.88,-0.62);

\node[anchor=east] at (-0.35,-0.1) {$v_2$};
\node[anchor=west] at (0.34,0.38) {$v_1=v_3$};
\end{scope}

\node at (5.20,-1.38) {Move (8)};
\end{scope}

\end{tikzpicture}

\caption{
Several grove reduction moves of \cite{gao2025electrical},
modified to Definition~\ref{defn:groves_moves}. Each diagram shows a local modification
inside a local disk. The external edges indicate arbitrary incident edges of the
double grove; when two vertices are identified, the identified vertex retains
the union of the incident edges of the two original vertices. In our
boundary-preserving version, any summand that identifies two distinct boundary
vertices is omitted. Thus, in Moves~(3) and~(4), contracting an edge whose
endpoints are both boundary vertices gives zero. In Move~(8), if exactly two
vertices of the triangle are boundary vertices, the term identifying those two
vertices is omitted; if all three vertices are boundary vertices, the reduction
is zero.
}
\label{fig:grove-reduction-moves}
\end{figure}

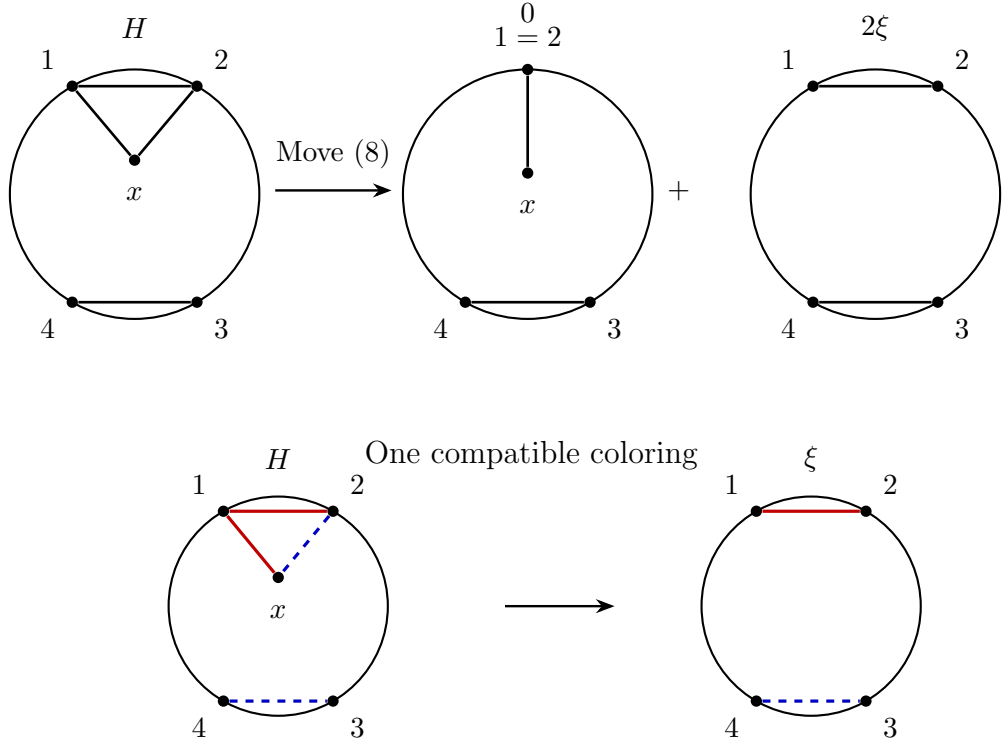
\begin{figure}[t]
\centering
\begin{tikzpicture}[
    every node/.style={font=\small},
    bvertex/.style={circle,fill=black,inner sep=1.5pt},
    ivertex/.style={circle,fill=black,inner sep=1.5pt},
    plain/.style={draw=black,line width=1pt},
    arr/.style={->,>=Stealth,line width=0.9pt},
    redlayer/.style={draw=red!75!black,line width=1.2pt},
    bluelayer/.style={draw=blue!75!black,line width=1.2pt,dashed}
]

\def\R{1.65}
\def\r{1.45}


\begin{scope}[xshift=0cm,yshift=0cm]
    \node at (0,2.20) {$H$};

    \draw[line width=0.8pt] (0,0) circle (\R);

    \node[bvertex] (h1) at (120:\R) {};
    \node[bvertex] (h2) at ( 60:\R) {};
    \node[bvertex] (h3) at (300:\R) {};
    \node[bvertex] (h4) at (240:\R) {};

    \node[above left=1pt of h1] {$1$};
    \node[above right=1pt of h2] {$2$};
    \node[below right=1pt of h3] {$3$};
    \node[below left=1pt of h4] {$4$};

    \node[ivertex] (hx) at (0,0.45) {};
    \node[below=4pt of hx] {$x$};

    \draw[plain] (h1)--(h2);
    \draw[plain] (h1)--(hx);
    \draw[plain] (h2)--(hx);
    \draw[plain] (h4)--(h3);
\end{scope}

\draw[arr] (1.86,0.05) -- (3.40,0.05)
    node[midway,above=4pt] {Move (8)};

\begin{scope}[xshift=5.20cm,yshift=0cm]
    \node at (0,2.40) {$0$};

    \draw[line width=0.8pt] (0,0) circle (\R);

    \node[bvertex] (m12) at (90:\R) {};
    \node[above=2.5pt of m12] {$1=2$};

    \node[bvertex] (m3) at (300:\R) {};
    \node[bvertex] (m4) at (240:\R) {};

    \node[below right=1pt of m3] {$3$};
    \node[below left=1pt of m4] {$4$};

    \node[ivertex] (mx) at (0,0.28) {};
    \node[below=4pt of mx] {$x$};

    \draw[plain] (m12)--(mx);
    \draw[plain] (m4)--(m3);
\end{scope}

\node at (7.20,0.05) {$+$};

\begin{scope}[xshift=9.80cm,yshift=0cm]
    \node at (0,2.20) {$2\xi$};

    \draw[line width=0.8pt] (0,0) circle (\R);

    \node[bvertex] (r1) at (120:\R) {};
    \node[bvertex] (r2) at ( 60:\R) {};
    \node[bvertex] (r3) at (300:\R) {};
    \node[bvertex] (r4) at (240:\R) {};

    \node[above left=1pt of r1] {$1$};
    \node[above right=1pt of r2] {$2$};
    \node[below right=1pt of r3] {$3$};
    \node[below left=1pt of r4] {$4$};

    \draw[plain] (r1)--(r2);
    \draw[plain] (r4)--(r3);
\end{scope}


\node at (5.25,-3.45) {\normalsize One compatible coloring};

\begin{scope}[xshift=1.90cm,yshift=-5.45cm]
    \node at (0,1.95) {$H$};

    \draw[line width=0.8pt] (0,0) circle (\r);

    \node[bvertex] (c1) at (120:\r) {};
    \node[bvertex] (c2) at ( 60:\r) {};
    \node[bvertex] (c3) at (300:\r) {};
    \node[bvertex] (c4) at (240:\r) {};

    \node[above left=1pt of c1] {$1$};
    \node[above right=1pt of c2] {$2$};
    \node[below right=1pt of c3] {$3$};
    \node[below left=1pt of c4] {$4$};

    \node[ivertex] (cx) at (0,0.38) {};
    \node[below=4pt of cx] {$x$};

    \draw[redlayer]  (c1)--(c2);
    \draw[redlayer]  (c1)--(cx);

    \draw[bluelayer] (c2)--(cx);
    \draw[bluelayer] (c4)--(c3);
\end{scope}

\draw[arr] (4.90,-5.45) -- (6.35,-5.45);

\begin{scope}[xshift=8.95cm,yshift=-5.45cm]
    \node at (0,1.95) {$\xi$};

    \draw[line width=0.8pt] (0,0) circle (\r);

    \node[bvertex] (d1) at (120:\r) {};
    \node[bvertex] (d2) at ( 60:\r) {};
    \node[bvertex] (d3) at (300:\r) {};
    \node[bvertex] (d4) at (240:\r) {};

    \node[above left=1pt of d1] {$1$};
    \node[above right=1pt of d2] {$2$};
    \node[below right=1pt of d3] {$3$};
    \node[below left=1pt of d4] {$4$};

    \draw[redlayer]  (d1)--(d2);
    \draw[bluelayer] (d4)--(d3);
\end{scope}

\end{tikzpicture}

\caption{A boundary triangle reduction and a compatible coloring.
In the top row, the triangular face of \(H\) has boundary vertices \(1,2\)
and interior vertex \(x\). By the boundary-adapted version of Move~(8), the
summand identifying the two boundary vertices \(1\) and \(2\) is omitted.
The other two identifications give the same reduced diagram \(\xi\), so
\(
\alpha(H)=2\xi.
\)
The bottom row shows one compatible coloring of \(H\) into two grove layers
and the induced coloring of \(\xi\). In this example, the red layer pairs
\(1\) with \(2\), the blue layer pairs \(3\) with \(4\), and each boundary
vertex is a singleton in the other layer.}
\label{fig:boundary_triangle_colored}
\end{figure}

\begin{defn}\label{defn:excess}
Let $H$ be a graph with node set $N$. For $v\in V(H)$, let
$d_H(v)$ denote the degree of $v$, with edges counted with multiplicity.

We define the \emph{excess} of $H$ by
\[
\operatorname{exc}(H)
:=
\sum_{v\in V_{\mathrm{int}}(H)}
\bigl(d_H(v)-4\bigr)
+
\sum_{i\in N}
\bigl(d_H(i)-1\bigr).
\]
\end{defn}

\begin{lem}\label{lem:excess}
The excess is preserved under the grove reduction moves in
Definition~\ref{defn:groves_moves}. More precisely, suppose that
\[
\alpha(H)\longmapsto \sum_{\beta} c_{\beta}\alpha(H_{\beta}).
\]
where the right-hand side denotes the intermediate sum produced by several reduction moves. Then, for every $\beta$ with $c_{\beta}\neq 0$,
\(
\operatorname{exc}(H_{\beta})
=
\operatorname{exc}(H).
\)
\end{lem}

\begin{proof}
We have
\(
\sum_{v\in V(H)}d_H(v)=2|E(H)|.
\)
Hence
\begin{align*}
\operatorname{exc}(H)
&=
\sum_{v\in V_{\mathrm{int}}(H)}
\bigl(d_H(v)-4\bigr)
+
\sum_{i\in N}
\bigl(d_H(i)-1\bigr)\\
&=
2|E(H)|-4|V_{\mathrm{int}}(H)|-|N|.
\end{align*}

Consider any nonzero term produced by a grove algebra reduction move $(3), (4), (6), (7)$ and  $ (8)$.
Every move which changes the graph removes one interior vertex and
decreases the number of edges, counted with multiplicity, by two.
Thus
\[
\Delta |E|=-2,
\qquad
\Delta |V_{\mathrm{int}}|=-1.
\]
Moreover, by the modified boundary rules, no nonzero term identifies
distinct boundary vertices, so $|N|$ is unchanged (though an interior vertex may be identified to a boundary vertex). Therefore
\[
\Delta\operatorname{exc}
=
2(-2)-4(-1)=0.
\]
Thus the excess is preserved under each nonzero reduction move, and hence
under any sequence of reductions.
\end{proof}

\begin{prop} \label{prop:reduction_3_matching}
Let $\mathscr P=(R_1,S_1,R_2,S_2)$ be an ordered balanced
partition of $N$, and let $(F_1,F_2)$ be a double grove
compatible with $\mathscr P$. Then, every complete
grove reduction of $F_1\cup F_2$ is a linear
combination of $3$-noncrossing matching graphs, and all $3$-noncrossing matching graphs are reduced. 
\end{prop}

\begin{proof}
Let $H=F_1\cup F_2$, with edges counted with multiplicity. Since each
$F_\ell$ is a spanning forest,
\[
\sum_{v\in V}\deg_{F_\ell}(v)=2|V|-2c_\ell,
\]
where $c_\ell$ is the number of components of $F_\ell$.

If $p_\ell$ and $s_\ell$ denote the numbers of paired and singleton boundary
blocks of $F_\ell$, then
\(
c_\ell=p_\ell+s_\ell.
\)
Since every boundary vertex is paired in exactly one of the two groves
and is a singleton in the other,
\(
p_1+p_2=m,
s_1+s_2=2m,
\)
and hence
\(
c_1+c_2=3m.
\)
Therefore
\[
\sum_{v\in V} d_H(v)
=
4|V|-6m.
\]
Since $|N|=2m$, it follows that
\[
\operatorname{exc}(H)
=
\sum_{v\in V_{\mathrm{int}}}(d_H(v)-4)
+
\sum_{i\in N}(d_H(i)-1)
=0.
\]

By the preceding lemma, every nonzero term $H'$ obtained after reduction
also has
\(
\operatorname{exc}(H')=0.
\)
In a nonzero reduced term, every interior vertex has degree at least $4$;
if it had degree at most $3$, then one of Rules (3)--(7) would apply.
Every boundary vertex has degree at least $1$, by Rule (2). Thus, every summand in
\[
\operatorname{exc}(H')
=
\sum_{v\in V_{\mathrm{int}}(H')}(\deg_{H'}(v)-4)
+
\sum_{i\in N}(\deg_{H'}(i)-1)
\]
is nonnegative. Since their sum is zero, each must vanish. Hence every
interior vertex of $H'$ has degree $4$ and every boundary vertex has
degree $1$.

Finally, the reduction move $(8)$ eliminates triangular faces. Thus, every
nonzero reduced term has $1$-valent boundary vertices, $4$-valent
interior vertices, and no triangular faces. 

Following opposite edges through each interior vertex partitions the
edges uniquely into strands. Each strand either joins two boundary
vertices or is closed. We claim that
\begin{enumerate}
\item[$(i)$] there are no closed strands;
\item[$(ii)$] any two distinct strands cross at most once;
\item[$(iii)$] no three strands pairwise cross; and
\item[$(iv)$] no strand crosses itself.
\end{enumerate}

Suppose that one of \((i)\)--\((iv)\) fails. Then the forbidden strand
configuration contains a simple closed curve \(\gamma\), formed from strand
segments, which bounds a simply connected region \(R\) and has at most
three corners. By a corner of \(\gamma\), we mean a vertex at which the two
edges of \(\gamma\) are adjacent in the planar cyclic order, rather than
opposite. Such a curve can be chosen with no corners in the presence of a
closed strand, with one corner in the presence of a self-intersection, with
two corners if two strands cross more than once, and with three corners if
three strands pairwise cross. Choose \(R\) innermost among all such regions.

Consider the subgraph consisting of \(\gamma\) together with all vertices
and edges contained in \(R\). Let \(v,e,f,c\) denote its numbers of
vertices, edges, bounded faces, and connected components, respectively,
and let \(b\) be the number of vertices lying on \(\gamma\). Every vertex
strictly inside \(R\) has degree \(4\). A noncorner vertex on \(\gamma\)
has degree at least \(3\), while at each corner, the degree is at least $2$. Since there are \(k\leq 3\) corners,
\[
2e\geq 4(v-b)+3(b-k)+2k
=4v-b-k.
\]

Euler's formula gives
\(
v-e+f=c.
\)
Moreover, every interior edge is incident to two bounded faces, whereas
each edge of \(\gamma\) is incident to only one. Since \(\gamma\) has
\(b\) edges,
\[
\sum_Q \deg(Q)=2e-b,
\]
where the sum is over the bounded faces \(Q\). Therefore
\[
\sum_Q \bigl(4-\deg(Q)\bigr)
=4f-2e+b
=4c-4v+2e+b
\geq 4c-k>0.
\]
The last inequality follows since $c\geq 1,k\leq 3$. Hence, at least one bounded face has degree at most \(3\), contradicting
reducedness. Thus \((i)\)--\((iv)\) all hold.

Hence, the resulting matching is
$3$-noncrossing that uses all edges in the reduced graph. One can verify that none of the reduction moves apply to a $3$-noncrossing matching graph, and hence it is reduced.
\end{proof}

\begin{prop}\label{prop:coloring_proposition}
Let \(H\) be an uncolored double-grove superposition and write
\[
\alpha(H)=\sum_{\xi\in\mathrm{TC}(2m)}a_\xi\,\xi.
\]
Recall the definition of $T_H$ (Definition \ref{defn:invariant_assigned}).
Then
\[
T_H
=
\sum_{\xi\in\mathrm{TC}(2m)}
a_\xi\,T_\xi.
\]
\end{prop}

\begin{proof}
It suffices to verify that each reduction move in
Definition~\ref{defn:groves_moves} preserves the admissible
signed coloring count for a fixed boundary tensor $v_{\mathscr P}$.

The zero rules admit no compatible grove decompositions.
Indeed, a loop cannot belong to either forest, and an edge of
multiplicity at least three forces a repeated edge in one layer.
An isolated boundary vertex cannot be in a grove component with another boundary vertex.
An interior vertex of degree $1$ can only be incident to an edge from one of the two layers.

For distinct parallel edges, each layer can contain at most one,
since two would form a cycle. Thus three or more parallel edges
admit no compatible decomposition. For exactly two parallel edges,
one belongs to each layer. Contracting them gives two lifts of
each compatible coloring of the reduced graph, accounting for
the factor $2$ in rule~(4). If their endpoints are both boundary
vertices, they would be in the same grove component in both layers, which is
forbidden. The same observation applies to a double edge
joining two boundary vertices.

For a double edge, the two copies must belong to the two different
grove layers. Contracting the double edge therefore gives a bijection
between admissible colorings before and after the contraction.

If an interior vertex has degree $2$, then its two incident edges must lie
in different grove layers. Any admissible coloring of the graph excluding the two incident edges, extends in exactly two ways, depending on which incident
edge is assigned to which layer. This accounts for the factor
$2$ in Rule~(6).

If an interior vertex has degree $3$, then both grove layers must contain
an edge incident to the vertex. Hence, exactly two
of the three incident edges have the same color.
The three possible choices of this pair correspond precisely to the
three graphs \(H_1,H_2,H_3\) appearing in the reduction rule. Hence
\[
T_H(v_{\mathscr P})
=
T_{H_1}(v_{\mathscr P})+T_{H_2}(v_{\mathscr P})+T_{H_3}(v_{\mathscr P}).
\]

For the triangular-face reduction, every admissible coloring
has two edges of one color and one edge of the other color.
Contract the uniquely colored edge and replace the remaining
two edges with a single edge of their color. This preserves the boundary connectivity
in each layer. There are $3$ choices for the uniquely colored edge corresponding to the reduction terms.  Terms identifying two
boundary vertices admit no compatible colorings since those
vertices would be in the same connected component in both grove layers. If all three vertices are boundary vertices, the two
same-colored edges connect three boundary vertices,
so no admissible coloring exists.

Each local correspondence preserves the boundary connectivity
in each grove layer. Hence, it preserves the permutations
\(\rho_1(c)\) and \(\rho_2(c)\), and therefore the sign
\(\operatorname{sgn}_{\mathscr P}(c)\).
The coloring identities thus hold for the signed sums as well.

Thus, every single move reduction
\(
H\longmapsto \sum_\beta a_\beta H_\beta
\)
satisfies
\(
T_H(v_{\mathscr P})
=
\sum_\beta a_\beta T_{H_\beta}(v_{\mathscr P}).
\)
Since the equality holds over the sequence of reductions defining \(\alpha(H)\)
we have that
\[
T_H(v_{\mathscr P})
=
\sum_{\xi\in\mathrm{TC}(2m)}
a_\xi T_\xi(v_{\mathscr P}),
\]
\end{proof}

We now give a precise formulation of Theorem~\ref{thm:sp4_grove_expansion1}.
\begin{thm}\label{thm:sp4_grove_expansion}
Let \(G\) be a planar electrical network with \(2m\) boundary vertices.
Then
\[
Z_G
=
\frac{1}{Z_{\mathrm{unc}}^2}
\sum_{\xi\in \mathrm{TC}(2m)}
Z[\xi]\,T_\xi,
\]
where
\[
Z[\xi]
:=
\sum_H
w(H)\,[\xi]\alpha(H),
\qquad
w(H)
:=
\prod_{e\in E(G)} c_e^{\mu_H(e)}.
\]
Here the sum is over distinct uncolored superpositions
\(H=F_1\cup F_2\) arising from the double-grove construction above,
with edges counted with multiplicity, and
\(\mu_H(e)\in\{0,1,2\}\) is the multiplicity of \(e\) in \(H\).
Each superposition \(H\) is counted once, independently of the number of
compatible decompositions \(H=F_1\cup F_2\), and
\([\xi]\alpha(H)\) denotes the coefficient of \(\xi\) in its
grove reduction.
\end{thm}

\begin{proof}
It suffices to verify the identity on basis tensors with factors
in \(\{e_1,e_2,f_1,f_2\}\). Both sides vanish on unbalanced
basis tensors, so fix an ordered balanced partition
\(\mathscr P=(R_1,S_1,R_2,S_2)\) of \(N\).

By Lemma~\ref{lem:sp4_test_evaluation} and
Theorem~\ref{thm:grove_expansion}, grouping ordered grove
pairs by their uncolored superposition gives
\[
Z_{\mathrm{unc}}^2Z_{G}(v_{\mathscr P})
=
\sum_H w(H)
\sum_{c\in\mathrm{Col}(H;\mathscr P)}
\operatorname{sgn}_{\mathscr P}(c)
=
\sum_H w(H)\,T_H(v_{\mathscr P}).
\]
Here \(H\) ranges over the superpositions in the theorem,
each counted once. If \(H\) admits no coloring compatible
with \(\mathscr P\), its contribution is zero.

By Proposition~\ref{prop:coloring_proposition},
\[
T_H(v_{\mathscr P})
=
\sum_{\xi\in\mathrm{TC}(2m)}
[\xi]\alpha(H)\,T_\xi(v_{\mathscr P}).
\]
Substituting and changing the order of the sums yields
\[
\begin{aligned}
Z_{\mathrm{unc}}^2Z_{G}(v_{\mathscr P})
&=
\sum_H w(H)
\sum_{\xi\in\mathrm{TC}(2m)}
[\xi]\alpha(H)\,T_\xi(v_{\mathscr P})\\
&=
\sum_{\xi\in\mathrm{TC}(2m)}
\left(\sum_H w(H)[\xi]\alpha(H)\right)
T_\xi(v_{\mathscr P})\\
&=
\sum_{\xi\in\mathrm{TC}(2m)}
Z[\xi]\,T_\xi(v_{\mathscr P}).
\end{aligned}
\]
The identity therefore holds on every basis tensor, and
hence on \(W^{\otimes 2m}\) by linearity. Dividing by
\(Z_{\mathrm{unc}}^2\) proves the theorem.
\end{proof}

\begin{cor}\label{cor:sp4_basis}
The tensors
\(
\{T_\xi:\xi\in\mathrm{TC}(2m)\}
\)
form a basis of
\(
U=\operatorname{Hom}_{\Sp(4)}
\bigl(W^{\otimes 2m},\mathbb C\bigr).
\)
\end{cor}

\begin{proof}
Let
\(
V:=\operatorname{span}\{T_\xi:\xi\in\mathrm{TC}(2m)\}.
\)
By Lemma~\ref{lem:grove_sp4}, \(V\subseteq U\).

Suppose that \(v\in W^{\otimes 2m}\) satisfies
\(
T_\xi(v)=0\) for every \(\xi\in\mathrm{TC}(2m).
\)
By Theorem~\ref{thm:sp4_grove_expansion},
\(
Z_G(v)=0
\)
for every circular planar electrical network \(G\).
On the other hand, Lemma~\ref{lem:pairing_basis} gives
\[
0
=
Z_G(v)
=
\sum_{\pi\in M(2m)}
(-1)^{\operatorname{cr}(\pi)+m}
T_\pi(v)
\prod_{\{i,j\}\in\pi}L_{ij}.
\]
The space of response matrices of circular planar networks contains a nonempty open subset of
\(
\mathbb R^{\binom{2m}{2}}
\)
\cite[Section~5.3]{curtis2000inverse}. Hence, the polynomial above vanishes in the entries
\(L_{ij}\) and \(i<j\).
Since distinct perfect matchings give distinct monomials,
\(
T_\pi(v)=0
\) for every \(\pi\in M(2m).
\)

By the first fundamental theorem for the symplectic group
\cite[Theorem~2.2]{rubey2015combinatorics}, the pairing tensors
\(T_\pi\) span \(U\).  Therefore every element of \(U\) vanishes on \(v\).
We have shown that
\(
V^\perp=U^\perp.
\)
Since \(V\subseteq U\), we have that \(V=U\).

Finally,
\(
\dim U=|\mathrm{TC}(2m)|
\)
by \cite[Theorems~2.13 and~2.18]{rubey2015combinatorics}.
Thus, the spanning family \(\{T_\xi\}\) has the correct size and is
therefore a basis.
\end{proof}

In particular, the reduction $\alpha(H)$ is independent of the
chosen reduction sequence. Indeed, if two sequences produce
$\sum_\xi a_\xi\xi$ and $\sum_\xi b_\xi\xi$, then
Proposition~\ref{prop:coloring_proposition} gives
\[
\sum_\xi a_\xi T_\xi = T_H = \sum_\xi b_\xi T_\xi.
\]
Since the tensors $T_\xi$ are linearly independent, we have
$a_\xi=b_\xi$ for every $\xi$.

\section*{Acknowledgments}
The author thanks his advisor, Richard Kenyon, for helpful
discussions and suggestions. The author used ChatGPT to assist with drafting and revising the manuscript, providing additional correctness checks of some mathematical arguments, and preparing figures. ChatGPT was not used to generate any novel mathematical ideas or results presented in the paper.

\bibliographystyle{alpha}
\bibliography{refs}
\end{document}